\documentclass[11pt,a4paper]{amsart}

\usepackage{amsmath,amscd}
\usepackage{amssymb}
\usepackage{amsthm}
\usepackage[nobysame, numeric]{amsrefs}

\usepackage[ocgcolorlinks, linkcolor=blue]{hyperref}

\usepackage[nobysame, numeric]{amsrefs}  

\usepackage{graphicx}

\usepackage{mathrsfs}

\usepackage{calc}
             {\begin{list}{\arabic{enumi}.}{\usecounter{enumi}%
              \setlength{\labelsep}{0.5em}%
              \settowidth{\labelwidth}{(\arabic{enumi})}%
              \setlength{\leftmargin}{\labelwidth+\labelsep}}}%
             {\end{list}}

\usepackage{comment}

\DeclareRobustCommand{\SkipTocEntry}[5]{}

\usepackage{xcolor}
\definecolor{LOcolor}{RGB}{150,100,0}

\newtheorem{Theorem}{Theorem}[section]

\newtheorem{Lemma}[Theorem]{Lemma}

\theoremstyle{definition}
\newtheorem{Definition}[Theorem]{Definition}

\newtheorem{Remark}[Theorem]{Remark}

\numberwithin{equation}{section}

\newcommand{\mR}{\mathbb{R}}                    
\newcommand{\abs}[1]{\lvert #1 \rvert}          
\newcommand{\norm}[1]{\lVert #1 \rVert}         

\newcommand{\ol}[1]{\overline{#1}}

\newcommand{\supp}{\mathrm{supp}}

\newcommand{\eps}{\varepsilon}

\newcommand{\p}{\partial}

\newcommand{\tbl}{\textcolor{blue}}

\newcommand{\mikko}[1]{\begin{quotation}[\textbf{\color{blue}Mikko's comment:\
		}{\color{blue}\textit{#1}}]\end{quotation}}
\newcommand{\henrik}[1]{\begin{quotation}[\textbf{\color{red}Henrik's comment:\
		}{\color{red}\textit{#1}}]\end{quotation}}

\newcounter{sidenote}

\begin{document}

\title{A free boundary approach to non-scattering obstacles with vanishing contrast}

\author{Mikko Salo}
\address{Department of Mathematics and Statistics, P.O. Box 35 (MaD), FI-40014 University of Jyv\"{a}skyl\"{a}, Finland.}
\email{\href{mailto:mikko.j.salo@jyu.fi}{mikko.j.salo@jyu.fi}}

\author{Henrik Shahgholian}
\address{Department of Mathematics, KTH Royal Institute of Technology, SE-100 44 Stockholm, Sweden.}
\email{\href{mailto:henriksh@kth.se}{henriksh@kth.se}}




\begin{abstract}
Motivated by questions in inverse scattering theory, we develop free boundary methods in obstacle problems where both the solution and the right hand side of the equation may have varying sign. The key condition that prevents the appearance of corners is that the right hand side should be related to a harmonic polynomial. In this setting we prove new free boundary results not found in existing literature. Notably, our results imply that piecewise $C^1$ or convex penetrable obstacles in two dimensions and edge points in higher dimensions always cause nontrivial scattering of any incoming wave.
\end{abstract}

\maketitle


\section{Introduction} \label{sec_introduction}

\subsection{Motivation}

In this work we study free boundary problems motivated by applications in inverse boundary problems and inverse scattering. One such application is a single measurement inverse problem for the equation $(\Delta+q)u = 0$ in $\Omega$, where $\Omega \subset \mR^n$ is a bounded open set with smooth boundary and $q \in L^{\infty}(\Omega)$ is such that the Dirichlet problem for $\Delta+q$ in $\Omega$ has a unique solution in $H^1(\Omega)$ for any boundary data in $H^{1/2}(\p \Omega)$. We suppose that we can prescribe one Dirichlet data $f$ on $\p \Omega$, and that we can measure the corresponding Neumann data $\p_{\nu} u|_{\p \Omega}$ where $u$ is the solution of 
\[
(\Delta+q)u = 0 \text{ in $\Omega$}, \qquad u|_{\p \Omega} = f.
\]

The inverse problem is to determine some properties of an unknown potential $q$ from the measurement $\p_{\nu} u|_{\p \Omega}$ corresponding to a fixed Dirichlet data $f$. This problem is closely related to the Calder\'on problem with a single measurement (see e.g.\ \cite{Alessandrini1999, LiuTsou2020} and references therein). An analogous single measurement problem arises in inverse scattering at a fixed frequency $k > 0$. There one would like to determine some properties of a compactly supported potential (or index of refraction) $q \in L^{\infty}(\mR^n)$ by measuring the far field pattern $u_q^{\infty}$ of the outgoing solution of 
\[
(\Delta+k^2+q) u = 0 \text{ in $\mR^n$}
\]
that corresponds to a fixed incoming wave $u_0$ solving $(\Delta+k^2)u_0 = 0$ in $\mR^n$. For a detailed discussion of these applications and for the definition of $u_q^{\infty}$ we refer the reader to \cite{SaloShahgholian}. 

In both cases above the measurement is a function on an $(n-1)$-dimensional manifold, hence it depends locally on $n-1$ variables, and it is expected that one can only recover partial information on the unknown function $q(x)$ that depends on $n$ variables. In particular one might hope to recover $D = \supp(q)$, i.e.\ the shape of the scattering obstacle. It has been observed that some obstacles are \emph{non-scattering}, or \emph{invisible}, in the sense that they may produce the same measurement as empty space (i.e.\ $q=0$ or $D = \emptyset$). On the other hand, many obstacles whose boundary has a singularity have been proved to scatter every incident wave nontrivially in the sense that $u_q^{\infty} \not\equiv 0$ for any incident wave $u_0 \not\equiv 0$.

This research on non-scattering frequencies was initiated in \cite{BlastenPaivarintaSylvester} where it was proved that obstacles with a rectangular corner point always scatter. There have since been several extensions and related single measurement results. The method of \cite{BlastenPaivarintaSylvester} was based on complex geometrical optics solutions coming from the Calder\'on problem and Laplace transforms. Related methods have been pursued e.g.\ in \cite{PaivarintaSaloVesalainen, HuSaloVesalainen, BlastenLiu1, BlastenLiu2}. Another interesting method based on boundary value problems in corner domains has been developed in \cite{ElschnerHu1, ElschnerHu2}. The case of planar obstacles with weakly singular boundary points has been discussed in \cite{LiHuYang2023}. These methods apply in cases where $\p D$ is assumed piecewise regular, and they are most complete when $n=2$ and become more limited even for $n=3$. Many further references may be found in \cite{SaloShahgholian, Liu_survey}.

The works \cite{CakoniVogelius, SaloShahgholian} introduced new methods from the theory of free boundary problems to deal with non-scattering problems (for the Calder\'on problem this was already done in \cite{AlessandriniIsakov}). These methods have the benefit that they apply in any dimension $n \geq 2$ and do not require that $\p D$ is piecewise regular, e.g.\ $D$ could be a Lipschitz domain or even a rather general open set. The results state that if $D$ is a non-scattering obstacle, then its boundary is a free boundary as in the (no-sign) obstacle problem, and hence locally either $\p D$ is very regular or it has thin complement. This implies for instance that if $\p D$ is Lipschitz but not $C^1$, then $D$ scatters every incident wave nontrivially. For precise results see \cite{CakoniVogelius, SaloShahgholian} and the related works \cite{VogeliusXiao2021, CakoniVogeliusXiao2023, KLSS2024, KSS2024a, KSS2024b}.

However, these results based on free boundary methods are in most cases only valid under a \emph{nonvanishing assumption}: the incoming wave $u_0$ should not vanish anywhere on $\p D$. Since real valued solutions of $(\Delta+k^2)u_0 = 0$ have many zeros in general (see e.g.\ \cite[Lemma 3.1]{SaloShahgholian}), this assumption is not guaranteed to hold in many cases of interest. The nonvanishing assumption is not required in the results mentioned above based on complex geometrical optics or corner domains. It would be highly desirable to remove this assumption also in the free boundary approach.

In our first main theorem we remove the nonvanishing assumption in the free boundary approach for certain two-dimensional domains. This will be done for potentials $q = h \chi_D$ where the contrast $h$ satisfies the following non-degeneracy condition at some $x_0 \in \p D$:
\begin{equation} \label{nondegen}
\text{there is $r > 0$ such that $h \in C^{\alpha}(\ol{D} \cap B_r(x_0))$ and $h(x_0) \neq 0$}.
\end{equation}

\begin{Definition}\label{piecewise-C1}(Piecewise $C^1$ domains)
We say that $D \subset \mR^2$ has \emph{piecewise $C^1$ boundary} if for any $x_0 \in \p D$, after a rotation and translation taking $x_0$ to $0$, and for some $r > 0$ one has 
\[
D \cap B_r(0) = \{ x_2 > \eta(x_1) \} \cap B_r(0)
\]
where $\eta \in C([-r,r])$, and both $\eta|_{[-r,0]}$ and $\eta|_{[0,r]}$ are $C^1$ functions up to the endpoints of the respective closed intervals.
\end{Definition}

\begin{Theorem} \label{thm_main_twodim}
Let $D \subset \mR^2$ be a bounded open set, and suppose that 
\begin{itemize}
\item 
$D$ is convex; or 
\item 
$D$ is  simply connected and has piecewise $C^1$ boundary.
\end{itemize}
Let $k > 0$, and suppose that $h \in L^{\infty}(D)$ satisfies the non-degeneracy condition \eqref{nondegen} at $x_0 \in \p D$. If $\p D$ is not $C^1$ near $x_0$, then the potential $q = h \chi_D$ scatters every incident wave non-trivially. In other words, for any incident wave $u_0 \not\equiv 0$ solving $(\Delta+k^2)u_0=0$, the far field pattern $u_q^{\infty}$ is not identically zero.
\end{Theorem}

The theorem states that convex or piecewise $C^1$ planar obstacles whose boundary has a singular (non-$C^1$) point always scatter. We recall that bounded convex domains always have Lipschitz boundary \cite[Corollary 1.2.2.3]{Grisvard}.
The case of piecewise $C^2$ planar domains was proved in \cite{ElschnerHu2}. Note that proving the theorem under \eqref{nondegen} is equivalent to proving it under Assumption (a) in \cite{ElschnerHu2}, see Remark \ref{rem_eh_assumption}. A similar result for piecewise real-analytic domains that are $C^{k-1}$ but not $C^k$ for some $k$ is given in \cite{LiHuYang2023}.

We obtain an analogous result for domains in $\mR^n$, $n \geq 3$, that contain edge singularities, see Theorem \ref{thm_main_ndim}.

\begin{Definition}(Edge points) \label{def_edge_point}
Let $D \subset \mR^n$, $n \geq 3$, be an open set. We say that $x_0 \in \p D$ is an \emph{edge point} if there is $\delta > 0$ and a $C^1$ diffeomorphism $\Phi: B_\delta(x_0) \to V$ onto an open set $V \subset \mR^n$ such that $\Phi(x_0) = 0$, $D\Phi(x_0)$ is an rotation matrix, and 
\[
\Phi(\ol{D} \cap B_\delta(x_0)) = (S \times \mR^{n-2}) \cap \Phi(B_\delta(x_0)),
\]
where $S = \{ (r \cos \theta, r \sin \theta) \,:\, r > 0,\ 0 \leq \theta \leq \theta_0 \}$ is a closed sector with angle $\theta_0 \in (0,2\pi) \setminus \{\pi\}$.
\end{Definition}

Examples of domains with edge points include curvilinear polyhedra in $\mR^3$ (see e.g.\ \cite{ElschnerHu2}), and more generally curved polytopes in $\mR^n$. For instance, suppose that $0 \in \p D$ and 
\[
D \cap B_r = \{ x \in B_r \,:\, f_1 > 0, \ldots, f_N > 0 \}
\]
where $f_j: \mR^n \to \mR$ are $C^1$ functions such that $f_j(0) = 0$, $\nabla f_j(0) \neq 0$, and $\frac{\nabla f_1(0)}{|\nabla f_1(0)|}, \ldots, \frac{\nabla f_N(0)}{|\nabla f_N(0)|}$ are distinct. If $N=2$ then $0$ is an edge point, and if $N \geq 3$ then $0$ is a (convex) polyhedral singular point. Then points on $\{ f_j = 0 \} \cap \{ f_k = 0\}$ near $0$ are edge points when $j \neq k$.

\begin{Theorem} \label{thm_main_ndim}
Let $D \subset \mR^n$, $n \geq 3$, be a bounded Lipschitz domain such that $\mR^n \setminus D$ is connected and $\p D$ contains an edge point $x_0 \in \p D$. Let $k > 0$, and let $h$ satisfy \eqref{nondegen} at $x_0$. Then the potential $q = h \chi_D$ scatters every incident wave nontrivially, i.e.\ for any incident wave $u_0 \not\equiv 0$ solving $(\Delta+k^2)u_0=0$ the far field pattern $u_q^{\infty}$ is not identically zero.
\end{Theorem}

This result was proved for piecewise $C^2$ edges when $n=3$ in \cite{ElschnerHu2} via boundary value problems in corner domains. Our proof, as explained above, is different and it implements the free boundary approach to nonscattering problems in the presence of incoming waves that may vanish on $\p D$. 
In particular, the method for proving Theorem \ref{thm_main_ndim} reduces matters to a two-dimensional problem as in Theorem \ref{thm_main_twodim} by Federer's dimension reduction argument (see \cite{Weiss1999} or \cite[Lemma 10.9]{Velichkov2023}).

\begin{Remark}
Theorem \ref{thm_main_ndim} can also be formulated for piecewise $C^1$ domains, provided one adopts a suitable generalization of Definition \ref{piecewise-C1} to higher dimensions. We refrain from presenting a technical definition here and instead leave it to the reader to choose an appropriate one.
What matters for our purposes is simply that any blow-up of the set $D$ at a boundary point contains an edge point.
\end{Remark}

\subsection{Free boundary problem}

Let us next explain the connection to free boundary problems and the reason for the nonvanishing assumption in the earlier works \cite{CakoniVogelius, SaloShahgholian}. We consider the case of a bounded domain $\Omega \subset \mR^n$ (the scattering case is analogous upon replacing $\Delta$ by $\Delta+k^2$, see \cite{SaloShahgholian}). We start with a solution of 
\[
(\Delta+q)u = 0 \text{ in $\Omega$}, \qquad u|_{\p \Omega} = g, 
\]
and a corresponding solution $u_0$ for the case $q=0$,
\[
\Delta u_0 = 0 \text{ in $\Omega$}, \qquad u_0|_{\p \Omega} = g.
\]

We suppose that $q$ is non-scattering for the particular Dirichlet data $g$. This means that $\p_{\nu} u|_{\p \Omega} = \p_{\nu} u_0|_{\p \Omega}$. We also let $D = \supp(q)$ and make the simplifying assumptions that $\ol{D} \subset \Omega$, $\Omega \setminus \ol{D}$ is connected, and $q = h \chi_D$, where $h \in C^{\alpha}(\ol{\Omega})$ satisfies $h(x) \neq 0$ for all $x \in \p D$ (thus the potential has a sharp jump at the boundary of the obstacle). By using unique continuation in $\Omega \setminus \ol{D}$, we see that $w := u-u_0$ satisfies 
\begin{equation} \label{eq_fbp_first}
(\Delta+q)w = f \chi_D \text{ in $\Omega$}, \qquad w = 0 \text{ outside $D$},
\end{equation}
where $f := -h u_0$. By elliptic regularity $w$ is at least $C^{1,\beta}$, for all $\beta < 1$.

If $q=0$, the equation in \eqref{eq_fbp_first} is precisely the one that arises in obstacle problems in free boundary theory. If $f > 0$ and $w \geq 0$ near the boundary point of interest, this is the classical obstacle problem. If $f > 0$ but $w$ can have varying sign, this is a \emph{no-sign obstacle problem} which is also well understood \cite{PSU}. If $u_0 \neq 0$ at the boundary point of interest, this situation applies, and the theory of the no-sign obstacle problem ensures that $\p D$ is a free boundary that is either regular or has thin complement.

However, in general the solution $u_0$ may vanish at the boundary point of interest, and in such a case both $f$ and $w$ can have varying sign. 
In some specific cases, such problems are  called \emph{unstable obstacle problems} and in these problems the free boundary may have corners (see e.g.\ \cite{And-Sh-W2012}). Other variations with $f$ vanishing can also be found in  \cite{Yeressian}. 
A typical example is the quadrant $D = \{ x \in \mR^2 \,:\, x_1, x_2 > 0 \}$ with $w = (x_1)_+^2 (x_2)_+^2$, which satisfies 
\begin{equation} \label{eq_fbp_second}
\Delta w = 2 |x|^2 \chi_D \text{ in $\mR^2$}, \qquad w = 0 \text{ outside $D$.}
\end{equation}

There is a special feature of the inverse scattering problem which may still prevent the appearance of corners. Namely, in \eqref{eq_fbp_first} the function $f$ has the special form $f = -h u_0$, where $h$ is nonvanishing and $u_0$ is a solution of an elliptic equation. If $0 \in \p D$ is the boundary point of interest, this implies that one has Taylor expansions 
\[
h = h(0) + O(|x|^{\alpha}), \qquad u_0 = H(x) + O(|x|^{m+1}),
\]
where $h(0) \neq 0$ and $H(x)$ is a \emph{harmonic} homogeneous polynomial of degree $m$. Thus after taking a suitable blowup at $0$, the problem \eqref{eq_fbp_first} will take the form 
\[
\Delta u = H \chi_D \text{ in $B_2$}, \qquad u = 0 \text{ outside $D$},
\]
where $H$ is a \emph{harmonic polynomial}. Note in contrast that in \eqref{eq_fbp_second}, where the free boundary has a corner, the polynomial $|x|^2$ is not harmonic.

\subsection{Main results in free boundary theory}

Our first main result shows that several of the steps appearing in the study of no-sign obstacle problems go through also in our setting. First we need a definition from geometric measure theory.

\begin{Definition}\label{def:blowup}(Blowups)
 For a function $v$ defined in $B_2$, we refer to the limit (when it exists) of the sequence
$$
v_{r_j,z}(x) := \frac{v(r_jx + z)}{r_j^k },
$$
as $r_j \to 0$, as a blowup limit of $v$ at $z \in B_{1}$, of order $k$. When $z=0$ we set 
$$
v_{r_j}(x):= v_{r_j,0}(x) = \frac{v(r_jx )}{r_j^k }.
$$
\end{Definition}

\begin{Theorem} \label{thm_fbp_steps}
Let $D \subset \mR^n$ be a Lipschitz domain such that $0 \in \p D$, let $f \in C^{\alpha}_{\mathrm{loc}}(B_2)$, and suppose that $u \in H^1_{\mathrm{loc}}(B_2)$ solves 
\[
\Delta u = f \chi_D \text{ in $B_2$}, \qquad u=0 \text{ outside $D$}.
\]
Let $m \geq 0$, and assume that $f = H + R$ where $H \not\equiv 0$ is a homogeneous polynomial of degree $m$ and $|R(x)| \leq C |x|^{m+\alpha}$. Then $u$ has the following properties.
\begin{itemize}
\item 
\emph{(Optimal regularity)} $u \in C^{1,1}_{\mathrm{loc}}(B_2)$ with 
\[
|u(x)| + |x| \,|\nabla u(x)| + |x|^2 \, |\nabla^2 u(x)| \leq C |x|^{m+2}.
\]
\item 
\emph{(Homogeneity of blowups)} If $v$ is any blowup limit of $u_r(x) = u(rx)/r^{m+2}$, then $v$ is homogeneous of degree $m+2$ and solves 
\[
\Delta v = H \chi_{ \{ v \neq 0 \} } \text{ in $B_1$.}
\]
\item 
\emph{(Nondegeneracy)} If $v$ is any blowup limit at $0$ of order $m+2$, then for any $\eps$ with $0 < \eps < 1$ there is $c_{\eps} > 0$ so that 
\[
\sup_{B(x, \eps \abs{x})} \abs{v} \geq c_{\eps} \abs{x}^{m+2} \text{ whenever $x \in \ol{D} \cap \ol{B}_{1/2}$}.
\]
\item 
\emph{(Weak flatness)} If the support of some blowup limit at $0$ is the half-space $x \cdot e \geq 0$, then for any $\delta > 0$ there is $r > 0$ such that $\p D \cap B_r \subset \{ |x \cdot e| \leq \delta r \}$.
\end{itemize}
\end{Theorem}

The above results are valid in any dimension and for any homogeneous polynomial $H$, not necessarily harmonic. We have seen in \eqref{eq_fbp_second} that for a general polynomial $H$ the free boundary can have corners. We would like to show that this cannot happen when $H$ is harmonic, which is the case in inverse scattering problems. This would correspond to the result that the support of every blowup limit is a half-space (here we are using the a priori assumption that $D$ is Lipschitz). Currently we can only prove this in two dimensions.

\begin{Theorem} \label{thm_blowup_classification_intro}
In the setting of Theorem \ref{thm_fbp_steps}, if we additionally assume that $H$ is harmonic and $n=2$, then the support of any blowup limit $v$ is a half space. The blowup limits have explicit form depending on $H$ (see Lemma \ref{lemma_half_space_twodim}).
\end{Theorem}

If $D$ is a general Lipschitz domain, we are at the moment not able to prove that blowup limits are unique (even for $n=2$ and $f(x) = x_1$) or that the free boundary is more regular. However, if we make stronger a priori assumptions on $D$, we can prove regularity of the free boundary.

Here is our main two-dimensional result, which leads to Theorem \ref{thm_main_twodim}.

\begin{Theorem}[Regularity of free boundary near piecewise $C^1$ or convex points] \label{thm_fb_2d}
Let $D \subset \mR^2$, and suppose
that $\partial D$ is piecewise $C^1$  with $0 \in \p D$. Suppose further  that $u \in H^1_{\mathrm{loc}}(B_2)$ solves 
\[
\Delta u = f \chi_D \text{ in $B_2$}, \qquad u=0 \text{ outside $D$},
\]
where $f = H + R \in C^{\alpha}_{\mathrm{loc}}(B_2)$ with $H \not\equiv 0$ a harmonic homogeneous polynomial of degree $m$ and $|R(x)| \leq C |x|^{m+\alpha}$. Then $\partial  D$ is $C^1$ near $0$.

Moreover, if in the above setting one assumes that $D$ is convex instead of piecewise $C^1$, and if one has decompositions $f = H + R$ as above for points $z \in \p D$ near $0$ (where $H$, $R$ and $m$ depend on $z$), then $\partial  D$ is $C^1$ near $0$.
\end{Theorem}

Our result for $n \geq 3$ states that the free boundary cannot contain edge points.

\begin{Theorem}[Regularity of free boundary near edge points] \label{thm_fb_polyhedron}
Let $D \subset \mR^n$ be an open set such that $0 \in \p D$ is an edge point, and suppose that $u \in H^1_{\mathrm{loc}}(B_2)$ solves 
\[
\Delta u = f \chi_D \text{ in $B_2$}, \qquad u=0 \text{ outside $D$},
\]
where $f = H + R \in C^{\alpha}_{\mathrm{loc}}(B_2)$ with $H \not\equiv 0$ a harmonic homogeneous polynomial of degree $m$ and $|R(x)| \leq C |x|^{m+\alpha}$. Then $D$ is $C^1$ near $0$.
\end{Theorem}

\subsection{Open problems}

There are several questions in free boundary theory that we are currently not able to answer when the right hand side has varying sign. To have a simple model problem let $D \subset \mR^2$ be a Lipschitz domain with $0 \in \p D$, consider the harmonic polynomial $H(x) = x_1$ of degree $1$, and suppose that $u \in C^{1,1}_{\mathrm{loc}}(B_2)$ solves 
\[
\Delta u = x_1 \chi_D \text{ in $B_2$}, \qquad u|_{B_2 \setminus \ol{D}} = 0.
\]
\begin{itemize}
\item 
Is there a unique blowup limit of $u_r(x) = u(rx)/r^3$?
\item 
Is it true that $\p D$ is $C^1$ near $0$?
\item 
If $\p D$ is known to be $C^1$ near $0$, can one show that $\p D$ is $C^{1,\alpha}$,  $C^{\infty}$ or  real-analytic near $0$?
\end{itemize}

All these questions are relevant for harmonic polynomials of degree $m \geq 2$ as well. They are also relevant in dimensions $n \geq 3$. In two dimensions we  characterized all blowup limits, and it would be of interest to have a characterization also when $n \geq 3$. Theorem \ref{thm_fb_polyhedron} implies that the support of a blowup limit cannot contain an edge point, but it does not exclude vertices.

\subsection*{Organization of article}

Section \ref{sec_introduction} is the introduction. In Section \ref{sec_regularity} we prove the $C^{1,1}$ regularity and decay rates for solutions of $\Delta u = f \chi_D$ with $u = 0$ outside $D$ when $f$ vanishes to given order. Section \ref{sec_blowup_solutions} studies blowup limits and proves their homogeneity by using a modified Weiss energy functional, as well as giving a classification of blowup solutions in two dimensions. In Section \ref{sec_nondegeneracy} we show nondegeneracy and weak flatness properties and regularity of the free boundary for convex domains. The main theorems in the introduction are proved in Section \ref{sec_fb_regularity}.

\section{$C^{1,1}$ regularity and optimal decay rate} \label{sec_regularity}

In this section we study the regularity and vanishing order for solutions of the equation 
\begin{equation} \label{regularity_equation_first}
(\Delta + q)u = f \chi_D \text{ in $B_1$}, \qquad u|_{B_1 \setminus \ol{D}} = 0,
\end{equation}
when $f$ is a suitable function with $\abs{f(x)} \leq C\abs{x}^m$ for some integer $m \geq 0$, and $D$ is a Lipschitz domain with $0 \in \p D$. Let us begin with some initial remarks on functions $u$ that satisfy 
\begin{equation} \label{deltau_vanishing_model}
\abs{\Delta u(x)} \leq C \abs{x}^m \text{ in $B_1$}.
\end{equation}
We are interested in conditions ensuring that $\abs{u(x)} \leq C \abs{x}^{m+2}$. Clearly some normalization is needed, since the condition \eqref{deltau_vanishing_model} does not change if we add a harmonic function to $u$. The following result was proved in \cite{CaffarelliFriedman}:

\begin{Lemma} \label{lemma_cf}
If $n=3$, $\beta > 0$ is a non-integer, and $u \in C^1(\ol{B}_1)$ satisfies 
\[
\abs{\Delta u(x)} \leq C \abs{x}^{\beta},
\]
then $u = P + \Gamma$ where $P$ is a harmonic polynomial of order $[\beta]+2$ and $\Gamma$ satisfies 
\[
\abs{\Gamma(x)} \leq C \abs{x}^{\beta+2}, \qquad \abs{\nabla \Gamma(x)} \leq C \abs{x}^{\beta+1}.
\]
\end{Lemma}

Lemma \ref{lemma_cf} fails when $\beta$ is an integer. In particular, the case $\beta = 0$ is related to the well-known fact that a bounded Laplacian $\Delta u$ does not, in general, imply that $u \in C^{1,1}$.  
A counterexample is given by  
\begin{equation} \label{counterexample_coneone} 
u(x) = h(x) | \log |x||^\alpha,
\end{equation}  
where $h$ is any $(m+2)$-homogeneous harmonic polynomial and $\alpha \leq 1$.  
In this case, one can verify that  
\[
|\Delta u(x)| \leq C |x|^m |\log |x||^{\alpha -1},
\]  
which tends to zero slightly faster than $|x|^m$ when $\alpha < 1$, and is bounded by $C|x|^m$ when $\alpha = 1$.  
Another example illustrating this phenomenon, for $\alpha = 1$, can be found in \cite[Exercise 4.9]{GilbargTrudinger}.

Given these remarks it is clear that in order to have $\abs{u(x)} \leq C \abs{x}^{m+2}$ for solutions of \eqref{regularity_equation_first}, we need to exploit the fact that $u$ vanishes in a sufficiently large set adjacent to $0$ (i.e.\ that $0$ is a free boundary point). When $m = 0$ this is just the standard $C^{1,1}$ regularity result for the no-sign obstacle problem \cite{AnderssonLindgrenShahgholian}. We will prove an analogous result for any $m$ in the case where $D$ is assumed to have Lipschitz boundary. First we show the correct vanishing order for $u$ and $\nabla u$ following an argument in \cite{KarpShahgholian}.

\begin{Lemma} \label{lemma_c11_first}
Let $q \in L^{\infty}_{\mathrm{loc}}(B_2)$, let $m \geq 0$ be an integer, and suppose that $u \in H^1_{\mathrm{loc}}(B_2)$ solves 
\[
(\Delta + q)u = f \chi_D \text{ in $B_2$}, \qquad u|_{B_2 \setminus \ol{D}} = 0,
\]
where $\abs{f(x)} \leq C \abs{x}^m$ a.e.\ in $B_2$ and where $D \subset \mR^n$ is a Lipschitz domain with $0 \in \p D$. Then 
\[
\abs{u(x)} + \abs{x} \, \abs{\nabla u(x)} \leq C \abs{x}^{m+2} \text{ in $B_{3/2}$}.
\]
\end{Lemma}

\begin{proof}
Since the right hand side $f \chi_D$ is in $L^{\infty}_{\mathrm{loc}}(B_2)$, Calder\'on-Zygmund estimates and Sobolev embedding imply that $u \in W^{2,p}_{\mathrm{loc}} \cap C^{1,\alpha}_{\mathrm{loc}}(B_2)$ for any $p < \infty$ and $\alpha < 1$. Thus it is enough to prove the required estimate in $B_1$.

Let $S_r = \sup_{\ol{B}_r} \,\abs{u}$. We start by proving that $S_r \leq C r^{m+2}$ for $r \in (0,1]$. We argue by contradiction and suppose that for any $j \geq 1$ there is $\rho_j \in (0,1]$ so that $S_{\rho_j} > j \rho_{j}^{m+2}$. Choosing $r_j \in (0,1]$ to be the supremum of all $\rho_j \in (0,1]$ with this property, we have 
\begin{equation} \label{sr_opposite}
S_{r_j} \geq j r_{j}^{m+2}, \qquad S_r \leq j r^{m+2} \text{ for $r \geq r_j$}.
\end{equation}
Note that $(r_j)$ is a nonincreasing sequence with $r_j \to 0$ as $j \to \infty$, since $r \mapsto r^{-m-2} S_r$ is continuous. Define the rescaled functions 
\[
u_j(x) = \frac{u(2 r_{j} x)}{j r_{j}^{m+2}}.
\]
Then 
\[
\sup_{\ol{B}_{1/2}} \,\abs{u_j} = \frac{S_{r_j}}{j r_{j}^{m+2}} \geq 1
\]
and 
\[
\sup_{\ol{B}_{1}} \,\abs{u_j} \leq \frac{S_{2 r_j}}{j r_{j}^{m+2}} \leq \frac{j(2r_j)^{m+2}}{j r_j^{m+2}} \leq 2^{m+2}.
\]
Moreover, since $\abs{f(x)} \leq C \abs{x}^m$, we have for $\abs{x} \leq 1$ 
\[
\abs{\Delta u_j(x)} \leq \frac{4}{j r_{j}^m} \abs{(f\chi_D)(2 r_{j} x) - (q u)(2 r_{j} x)} \leq \frac{C}{j} + C \frac{S_{2 r_j}}{j r_j^m} \leq \frac{C}{j} + C r_j^2.
\]

Now for any $\alpha < 1$, choosing $p < \infty$ large enough we may use Calder\'on-Zygmund estimates and Sobolev embedding to obtain that 
\[
\norm{u_j}_{C^{1,\alpha}(\ol{B}_{1/2})} \leq C \norm{u_j}_{W^{2,p}(\ol{B}_{1/2})} \leq C(\norm{u_j}_{L^p(B_1)} + \norm{\Delta u_j}_{L^p(B_1)}).
\]
Thus $(u_j)$ is uniformly bounded in $C^{1,\alpha}(\ol{B}_{1/2})$, and by compactness there is a subsequence $(u_{j_k})$ converging in $C^1(\ol{B}_{1/2})$ to some $v$. By the above estimates one has 
\[
\sup_{\ol{B}_{1/2}} \,\abs{v} \geq 1, \qquad \Delta v = 0 \text{ in $B_1$}.
\]
Moreover, since $D$ is a Lipschitz domain with $0 \in \p D$ and $u|_{B_2 \setminus \ol{D}} = 0$, it follows that there is an open cone $C$ in $\mR^n$ so that each $u_j$ and hence $v$ vanish in $C \cap B_{1/2}$. By unique continuation one has $v \equiv 0$ in $B_{1/2}$, but this is impossible since $\sup_{\ol{B}_{1/2}} \,\abs{v} \geq 1$. We have reached a contradiction and proved the estimate $\abs{u(x)} \leq C \abs{x}^{m+2}$ in $B_1$.

Next we define $T_r = \sup_{\ol{B}_r} \, \abs{\nabla u}$ and wish to prove the estimate $T_r \leq C r^{m+1}$. However, this follows by repeating the contradiction argument above where $S_r$ is replaced by $T_r$, the estimates for $\sup\,\abs{u_j}$ are replaced by estimates for $\sup\,\abs{\nabla u_j}$, and one uses additionally the simple estimate $\abs{u(x)} \leq (\sup\,\abs{\nabla u}) \abs{x}$.
\end{proof}

In fact we can use the same method to show the following stronger result, which will be used for proving $C^{1,1}$ regularity.

\begin{Lemma} \label{lemma_c11_second}
Let $q \in L^{\infty}_{\mathrm{loc}}(B_2)$, let $m \geq 0$ be an integer, and suppose that $u \in H^1_{\mathrm{loc}}(B_2)$ solves 
\[
(\Delta + q)u = f \chi_D \text{ in $B_2$}, \qquad u|_{B_2 \setminus \ol{D}} = 0,
\]
where $\abs{f(x)} \leq C \abs{x}^m$ and $D \subset \mR^n$ is a Lipschitz domain with $0 \in \p D$. Then in $B_1$ 
\[
\abs{u(x)} \leq C \abs{x}^{m} d(x, \p D)^2, \qquad \abs{\nabla u(x)} \leq C \abs{x}^{m} d(x, \p D).
\]
\end{Lemma}
\begin{proof}
We claim that there is $C > 0$ so that for any $r \in [0,1/4]$, any $z \in B_{2r} \cap \p D$ and any $s \in [0,r]$ one has 
\[
\sup_{\ol{B}(z,s)} \abs{u} \leq C r^m s^2.
\]
If this is not true, then for any $j \geq 1$ there exist $r_j$, $z^j \in B_{2r_j} \cap \p D$ and $s_j \in [0,r_j]$ such that (after replacing $s_j$ by the supremum of such values) 
\[
\sup_{\ol{B}(z^j, s_j)} \abs{u} \geq j r_j^m s_j^2, \qquad \sup_{\ol{B}(z^j, s)} \abs{u} \leq j r_j^m s^2 \text{ for $s \geq s_j$.}
\]
Define the rescaled function 
\[
u_j(x) = \frac{u(z^j + 2 s_j x)}{j r_j^m s_j^2}.
\]
Then we have 
\[
\sup_{\ol{B}_{1/2}} \,\abs{u_j} \geq 1, \qquad \sup_{\ol{B}_{1}} \,\abs{u_j} \leq 2^2,
\]
and for $\abs{x} \leq 1$ 
\[
\abs{\Delta u_j(x)} \leq \frac{4}{j} \frac{\abs{\Delta u(z^j + 2 s_j x)}}{r_j^m} \leq \frac{C}{j} \left[ \frac{\abs{u(z^j + 2 s_j x)}}{r_j^m} + \frac{\abs{f(z^j + 2 s_j x)}}{r_j^m} \right] \leq \frac{C}{j}
\]
since $\abs{z^j + s_j x} \leq 3 r_j$, $\abs{f(y)} \leq C \abs{y}^m$, and $\abs{u(y)} \leq C \abs{y}^{m+2}$ by Lemma \ref{lemma_c11_first}. As in the proof of Lemma \ref{lemma_c11_first}, we may extract a subsequence of $(u_j)$ that converges in $C^1(\ol{B}_{1/2})$ to some $v$ that satisfies $\Delta v = 0$ in $B_{1/2}$ and $\sup_{\ol{B}_{1/2}} \,\abs{v} \geq 1$, with $v$ vanishing in some open cone in $B_{1/2}$. The unique continuation principle gives $v = 0$, which is a contradiction.

Let us now derive the statement from the above claim. Let $x \in B_1$ and let $z$ be a point on $\p D$ with $\abs{x-z} = d(x, \p D)$. Fix $r = \abs{x}$ and $s = d(x, \p D)$. Since $0 \in \p D$ one has $s \leq r$ and  $\abs{z} \leq \abs{z-x} + \abs{x} = s + r \leq 2r$. We can now use the claim to obtain that 
\[
\abs{u(x)} \leq \sup_{\ol{B}(z,s)} \abs{u} \leq C r^m s^2 = C \abs{x}^m d(x, \p D)^2.
\]
The estimate for $\nabla u$ is proved in a similar way.
\end{proof}

Next we proceed to estimates for the second derivatives of $u$. It is well known (and follows using the counterexample \eqref{counterexample_coneone} above) that one needs some regularity for the function $f$ in order to have $u$ in $C^{1,1}$. It is sufficient that $f$ is H\"older or Dini continuous, and in fact the optimal condition is that $f = \Delta w$ for some $f \in C^{1,1}$ \cite{AnderssonLindgrenShahgholian}. We will give a version of this result for $C^{\alpha}$ functions $f$ that vanish to order $m \geq 1$ at the origin, in the sense that 
\begin{equation} \label{calpha_vanishing}
|f(x)| \leq C |x|^m, \qquad |f(x)-f(y)| \leq C \max(|x|,|y|)^{m-\alpha} |x-y|^{\alpha}.
\end{equation}
The following is the main result of this section.

\begin{Theorem} \label{thm_c11_regularity}
Let $q \in C^{\alpha}_{\mathrm{loc}}(B_2)$ for some $\alpha \in (0,1)$, and suppose that $u \in H^1_{\mathrm{loc}}(B_2)$ solves 
\[
(\Delta + q)u = f \chi_D \text{ in $B_2$}, \qquad u|_{B_2 \setminus \ol{D}} = 0,
\]
where $f \in C^{\alpha}_{\mathrm{loc}}(B_2)$ satisfies \eqref{calpha_vanishing} for $x, y \in \ol{B}_1$, 
and $D \subset \mR^n$ is a Lipschitz domain with $0 \in \p D$. Then $u \in C^{1,1}(\ol{B}_1)$ and 
\[
\abs{u(x)} + \abs{x} \, \abs{\nabla u(x)} + \abs{x}^2 \, \abs{\nabla^2 u(x)} \leq C \abs{x}^{m+2}.
\]
\end{Theorem}
\begin{proof}
Fix a point $x \in D \cap B_1$, write $r = \abs{x}$ and $d_x = d(x, \p D) \leq \abs{x} \leq r$, and consider the function 
\[
v(y) = \frac{u(x + d_x y)}{d_x^2}.
\]
By Lemma \ref{lemma_c11_second} we have for $\abs{y} \leq 1$ 
\begin{align*}
\abs{v(y)} &\leq \frac{C \abs{x + d_x y}^m (2d_x)^2}{d_x^2} \leq C r^m, \\
\abs{\nabla v(y)} &\leq \frac{C \abs{x + d_x y}^m (2d_x)}{d_x} \leq C r^m.
\end{align*}
Since $(\Delta+q)u = f$ in $D$ where $f \in C^{\alpha}$, we have from Schauder estimates that $u$ is $C^{2,\alpha}$ in $D$ and 
\begin{align*}
\norm{\nabla^2 v}_{L^{\infty}(B_{1/2})} &\leq C (\norm{v}_{L^{\infty}(B_1)} + \norm{\Delta v}_{C^{\alpha}(\ol{B}_1)}) \\
 &\leq C (r^m + \norm{(qu)(x + d_x \,\cdot\,)}_{C^{\alpha}(\ol{B}_1)} + \norm{f(x + d_x \,\cdot\,)}_{C^{\alpha}(\ol{B}_1)}).
\end{align*}
The second term on the right is $\leq C r^m d_x^2$ using the estimates for $\abs{v}$ and $\abs{\nabla v}$. For the last term, the estimate 
\[
\abs{f(x+d_x y) - f(x+d_x z)} \leq C r^{m-\alpha} d_x^{\alpha} \abs{y-z}^{\alpha} \leq C r^m \abs{y-z}^{\alpha}, \qquad y, z \in B_1,
\]
ensures that this term is $\leq C r^m$. Thus it follows that $\abs{\nabla^2 v(y)} \leq C r^m$ for $\abs{y} \leq 1/2$, and choosing $y=0$ gives $\abs{\nabla^2 u(x)} \leq C \abs{x}^m$.

We have proved that $u$ is $C^{1,1}_{\mathrm{loc}}$ in $D \cap B_2$ and satisfies 
\[
\abs{\nabla^2 u(x)} \leq C \abs{x}^{m}, \qquad x \in D \cap B_1.
\]
Since $u \in W^{2,p}_{\mathrm{loc}}(B_2)$ for any $p < \infty$, $u$ vanishes outside $D$ and $\p D$ has zero measure, the required estimate holds a.e.\ in $B_1$.
\end{proof}

We will apply the above result to right hand sides $f$ that are products of a $C^{\alpha}$ function and a smooth function vanishing to high order. Such functions have the following properties.

\begin{Lemma} \label{lemma_product_vanishing}
Let 
\[
f = P + R
\]
where $P$ is a homogeneous polynomial of order $m \geq 1$, and $R \in C^{\beta}(\ol{B}_1)$ satisfies $|R(x)| \leq C |x|^{m+\beta}$ for some $\beta > 0$. Then $f$ satisfies \eqref{calpha_vanishing} for $x, y \in \ol{B}_1$ with some $\alpha > 0$.

If $h \in C^{\alpha}(\ol{B}_1)$ and $v \in C^{m,\alpha}(\ol{B}_1)$ with $|v(x)| \leq C |x|^m$, then $f = hv$ has the above properties.
\end{Lemma}
\begin{proof}
If $f = P+R$ as above, then $|f(x)| \leq C |x|^m$. For $x, y \in \ol{B}_1$ one has 
\begin{align*}
|P(x)-P(y)| &\leq (\sup_{z \in [x,y]} |\nabla P(z)|) |x-y| \leq C \max(|x|,|y|)^{m-1} |x-y|, \\
|P(x)-P(y)| &\leq C \max(|x|,|y|)^m,
\end{align*}
and 
\begin{align*}
|R(x)-R(y)| &\leq C|x-y|^{\beta}, \\
|R(x)-R(y)| &\leq C \max(|x|,|y|)^{m+\beta}.
\end{align*}
Interpolating these estimates gives $|f(x)-f(y)| \leq C \max(|x|,|y|)^{m-\alpha} |x-y|^{\alpha}$ for some $\alpha > 0$. Thus $f$ satisfies \eqref{calpha_vanishing}.

Let $f = hv$ with $h$ and $v$ as in the statement. Since $\abs{v(x)} \leq C \abs{x}^m$, one has $\p^{\gamma} v(0) = 0$ for $\abs{\gamma} \leq m-1$ and the Taylor expansion for $v(x)$ at $0$ takes the form 
\[
v(x) = \sum_{\abs{\gamma} = m} \frac{\p^{\gamma} v(0)}{\gamma!} x^{\gamma} + \sum_{\abs{\gamma} = m} \frac{\abs{\gamma} x^{\gamma}}{\gamma!} \int_0^1 (1-t)^{m-1} (\p^{\gamma} v(tx) - \p^{\gamma} v(0)) \,dt.
\]
Since $v \in C^{m,\alpha}$, one has $v = P_1+R_1$ where $P_1$ is a polynomial of order $m$ and $|R_1(x)| \leq C |x|^{m+\alpha}$. Writing $h = h(0) + (h - h(0))$ gives the required decomposition $f = P + R$ with $P = h(0) P_1$.
\end{proof}

\section{Blowup solutions} \label{sec_blowup_solutions}

\subsection{Weiss monotonicity formula}

Suppose that we have a solution $u \in C^{1,1}_{\mathrm{loc}}$ of 
\begin{equation} \label{weiss_eq1}
(\Delta + q) u = f \chi_{\{ u \neq 0 \}} \text{ in $B_2$}
\end{equation}
where, as in Lemma \ref{lemma_product_vanishing},  
\begin{equation} \label{weiss_eq2}
\text{$f = H + R$, $H$ is homogeneous of degree $m$, and $\abs{R(x)} \leq C \abs{x}^{m+\alpha}$ for some $\alpha > 0$.}
\end{equation}
Also assume that $u$ satisfies the estimates 
\begin{equation} \label{weiss_eq3}
\abs{u(x)} + \abs{x} \, \abs{\nabla u(x)} + \abs{x}^2 \, \abs{\nabla^2 u(x)} \leq C \abs{x}^{m+2}.
\end{equation}
Recall the definition of blowup of order $k$ in Definition \ref{def:blowup},  and note that for  $k=m+2$, the function $u_r(x) = u(rx)/r^{m+2}$ also satisfies the estimates \eqref{weiss_eq3} and solves the equation 
\[
\Delta u_r(x) = (H(x) + r^{-m} R(rx) - r^2 q(rx) u_r(x)) \chi_{\{ u_r \neq 0 \}} \text{ in $B_2$}.
\]

We wish to study the possible blowup limits of $u_r$ when $r \to 0$. To this end we introduce a version of the Weiss energy functional (see \cite[\S3.5]{PSU}) adapted to our situation, namely 
\[
W(r, u) = W_H(r,u) = \frac{1}{r^{2m+n+2}}\int_{B_r} (\abs{\nabla u}^2 + 2 H u) \,dx - \frac{m+2}{r^{2m+n+3}} \int_{\p B_r} u^2 \,dS.
\]
This functional satisfies  
\[
W(rs, u) = W(s, u_r), \qquad 0 < r, s \leq 1.
\]
In particular $W(r,u) = W(1,u_r)$. We compute the derivative 
\begin{align}
\frac{1}{2} \p_r W(r,u) &= \frac{1}{2} \p_r W(1,u_r) \notag \\
 &= \int_{B_1} (\nabla u_r \cdot \nabla \p_r u_r + H \p_r u_r) \,dx - (m+2) \int_{\p B_1} u_r \p_r u_r \,dS \notag \\
 &= \int_{B_1} (-\Delta u_r + H) \p_r u_r \,dx + \int_{\p B_1} (\p_{\nu} u_r - (m+2) u_r) \cdot \p_r u_r \,dS \notag \\
 &= - \int_{B_1} r^{-m} R(rx) \p_r u_r \,dx + \int_{B_1} r^2 q(rx) u_r \p_r u_r \,dx + \int_{\p B_1} r (\p_r u_r)^2 \,dS. \label{weiss_energy_derivative}
\end{align}
In the last equality we used the facts that $(-\Delta u_r + H + r^{-m}R(rx)- r^2 q(rx) u_r)(\p_r u_r) = 0$ a.e.\ in $B_1$ and $\p_{\nu} u_r - (m+2) u_r = r \p_r u_r$ on $\p B_1$. The last term in \eqref{weiss_energy_derivative} is nonnegative. In order to obtain a monotone quantity, we introduce the functional 
\[
F(r,u) = 2 \int_0^r \int_{B_1} s^{-m} R(sx) \p_s u_s \,dx \,ds - 2 \int_0^r \int_{B_1} s^2 q(sx) u_s \p_s u_s \,dx \,ds.
\]
Then  
\[
\p_r \left[ W(r,u) + F(r,u) \right] = 2 \int_{\p B_1} r (\p_r u_r)^2 \,dS \geq 0.
\]

In order to study the limit as $r \to 0$, we observe that 
\[
\abs{F(r,u)} \leq C \int_0^r \int_{B_1} (s^{\alpha} \abs{x}^{m+\alpha} + s^2 \abs{u_s(x)}) \abs{\p_s u_s(x)} \,dx \,ds.
\]
Now $\abs{u_s(x)} \leq C \abs{x}^{m+2}$ and 
\[
\abs{\p_s u_s(x)} = \abs{-(m+2) s^{-m-3} u(sx) + s^{-m-2} x \cdot \nabla u(sx)} \leq C s^{-1} |x|^{m+2}.
\]
It follows that 
\[
\abs{F(r,u)} \leq C r^{\alpha}.
\]
On the other hand, one has 
\[
W(r,u) = W(1,u_r) \geq -C \int_{B_1} \abs{x}^{2m+2} \,dx - C \int_{\p B_1} \,dS \geq -C.
\]
Thus the nondecreasing quantity $W(r,u) + F(r,u)$ has a finite limit as $r \to 0$, and this limit equals $W(0+, u)$.

\begin{Definition}
Suppose $u \in C^{1,1}_{\mathrm{loc}}(B_2)$ satisfies \eqref{weiss_eq1}--\eqref{weiss_eq3}. We say that $v \in C^{1,1}(\ol{B}_1)$ is a blowup limit of $u$ is there is a sequence $r_j \to 0$ so that $u_{r_j} \to v$ in $C^1(\ol{B_1})$.
\end{Definition}

\begin{Lemma} \label{lemma_weiss_blowup}
Let $u \in C^{1,1}_{\mathrm{loc}}(B_2)$ satisfy \eqref{weiss_eq1}--\eqref{weiss_eq3}. Any blowup limit $v$ of $u$ is homogeneous of degree $m+2$ and solves the equation 
\[
\Delta v = H \chi_{\{ v \neq 0 \}} \text{ in $B_1$}.
\]
One has $W(s, v) = W(0+, u)$ for $0 < s \leq 1$. Moreover, the homogeneous degree $m+2$ extension of $v$ (still denoted by $v$) solves the equation 
\[
\Delta v = H \chi_{\{ v \neq 0 \}} \text{ in $\mR^n$}.
\]
\end{Lemma}
\begin{proof}
The proof is standard, see e.g.\ \cite[Lemma 16]{Yeressian}.
\end{proof}

\subsection{Blowups and global solutions}

It follows from Lemma \ref{lemma_weiss_blowup} that a blowup procedure at the point $0 \in \p D$ yields a function $v \in C^{1,1}_{\mathrm{loc}}(\mR^n)$ which is homogeneous of degree $m+2$ and solves the equation 
\[
\Delta v = H \chi_{\{ v \neq 0 \}} \text{ in $\mR^n$}
\]
where $H$ is a homogeneous polynomial of degree $m$. The case where $H$ is harmonic and $\mathrm{supp}(v)$ is a half space is of particular interest, since in this case  we expect $\p D$ to be regular near $0$.

We have the following result that characterizes half space solutions.

\begin{Lemma} \label{lemma_half_space}
The equation 
\[
\Delta v = f \text{ in $\mR^n_+$}, \qquad v|_{\{x_n=0\}} = \p_n v|_{\{x_n=0\}} = 0,
\]
has at most one solution in $H^1_{\mathrm{loc}}(\ol{\mR^n_+})$. If $f$ is a polynomial of order $m$, then there is a unique solution $v$ which is a polynomial of order $m+2$ given explicitly by 
\[
v(x', x_n) = \sum_{k=2}^{m+2} \frac{x_n^k}{k!} \left(  \sum_{j=0}^{[\frac{k-2}{2}]} \p_n^{k-2-2j} (-\Delta_{x'})^j f(x',0) \right).
\]
If $f$ is a homogeneous polynomial of order $m$, then $v$ is a homogeneous polynomial of order $m+2$.
\end{Lemma}
\begin{proof}
Uniqueness follows since any $v \in H^1_{\mathrm{loc}}(\ol{\mR^n_+})$ satisfying $\Delta v = 0$ in $\mR^n_+$ with $v|_{\{x_n=0\}} = \p_n v|_{\{x_n=0\}} = 0$ must vanish identically by the unique continuation principle. For existence, let $\mathcal{P}_k$ denote the space of polynomials of order $\leq k$. The operator 
\[
L: x_n^2 \mathcal{P}_m \to \mathcal{P}_m, \ \ Lv = \Delta v
\]
is linear and injective by the uniqueness argument above. Since it maps between spaces having the same finite dimension, $L$ is surjective. This shows existence of a solution. We have 
\[
v(x',x_n) = \sum_{k=0}^{m+2} \frac{\p_n^k v(x',0)}{k!} x_n^k.
\]
We can evaluate $\p_n^k v(x',0)$ by using the equation $\Delta v = f$ in $\{ x_n \geq 0 \}$ and the boundary conditions $v|_{\{x_n=0\}} = \p_n v|_{\{x_n=0\}} = 0$. The formula for $v$ follows.
\end{proof}

We can also compute the Weiss energy corresponding to a half space solution. It is notable that this energy only depends on  $m$, $n$ and the $L^2$ norm of $H$ on $S^{n-1}$. It does not depend on the particular half space where the solution is supported, or on the polynomial $H$ once its $L^2(S^{n-1})$ norm has been fixed.

\begin{Lemma} \label{lemma_halfspace_weiss_energy}
Let $H$ be a harmonic homogeneous polynomial of degree $m$, let $S$ be a half space in $\mR^n$, and suppose that $v \in C^{1,1}_{\mathrm{loc}}(\mR^n)$ is homogeneous of degree $m+2$ and solves 
\[
\Delta v = H \chi_S \text{ in $\mR^n$}, \qquad v|_{\mR^n \setminus S} = 0.
\]
Then 
\[
W_H(1, v) = \frac{1}{2(2m+n+2)} \frac{1}{\lambda_{m+2}-\lambda_m} \int_{S^{n-1}} H^2(\omega) \,d\omega.
\]
\end{Lemma}
\begin{proof}

First note that since $v$ is homogeneous of degree $m+2$, one has $\p_{\nu} v|_{\p B_1} = (m+2)v|_{\p B_1}$ and hence 
\[
W_H(1,v) = \int_{B_1} (\abs{\nabla v}^2 + 2Hv) \,dx - (m+2) \int_{\p B_1} v^2 \,dS = \int_{B_1} ((-\Delta v + 2H) v) \,dx = \int_{B_1} H v \,dx.
\]
Suppose that $e$ is a unit vector so that $S = \{ x \cdot e > 0 \}$, and let $R$ be a rotation matrix with $R^t e = e_n$. Then 
\[
\Delta(v \circ R) = (H \circ R) \chi_{\{ x_n > 0 \}}, \qquad v \circ R|_{\{x_n < 0\}} = 0.
\]
We also have 
\[
W_{H \circ R}(1, v \circ R) = W_H(1,v).
\]
Since the claimed formula for $W_H(1,v)$ is invariant under rotations of $H$, we may thus assume that $S = \{ x_n > 0 \}$.

Since $v$ is homogeneous of degree $m+2$, we can write $v(r\omega) = r^{m+2} \varphi(\omega)$ where $\omega \in S^{n-1}_+ = S^{n-1} \cap \{ x_n > 0 \}$ and $\varphi \in C^{1,1}(S^{n-1})$. Writing $\Delta = \p_r^2 + \frac{n-1}{r} \p_r + \frac{1}{r^2} \Delta_{S}$, where $\Delta_S$ is the Laplacian on $S^{n-1}$, and using the equation gives that 
\[
(\Delta_S + \lambda_{m+2}) \varphi(\omega) = H(\omega) \text{ in $S^{n-1}_+$}
\] 
where $\lambda_k = k(k+n-2)$. The condition $v|_{\{x_n < 0\}} = 0$ implies that $\varphi|_{\p S^{n-1}_+} = \p_{\nu} \varphi|_{\p S^{n-1}_+} = 0$. Denote by $\tilde{\varphi}$ the antipodally even/odd extension of $\varphi$ depending on whether $m$ is even/odd. Then $\tilde{\varphi} \in C^{1,1}(S^{n-1})$ and one has 
\[
(\Delta_S + \lambda_{m+2}) \tilde{\varphi}(\omega) = H(\omega) \text{ in $S^{n-1}$.}
\]
By elliptic regularity, $\tilde{\varphi}$ is in fact real-analytic. Expanding $\tilde{\varphi}(\omega)$ in terms of spherical harmonics and using that $H$ is a spherical harmonic of degree $m$, we obtain that 
\begin{equation} \label{varphitilde_formula}
\tilde{\varphi}(\omega) = \frac{1}{\lambda_{m+2}-\lambda_m} H(\omega) + Q(\omega)
\end{equation}
where $Q$ is a spherical harmonic of degree $m+2$.

We now compute 
\begin{align*}
W_H(1,v) &= \int_{B_1} H v \,dx = \int_0^1 \int_{S^{n-1}_+} r^{2m+2+n-1} H(\omega) \varphi(\omega) \,d\omega \,dr \\
 &= \frac{1}{2m+n+2} \int_{S^{n-1}_+} H(\omega) \varphi(\omega) \,d\omega = \frac{1}{2(2m+n+2)} \int_{S^{n-1}} H(\omega) \tilde{\varphi}(\omega) \,d\omega \\
 &= \frac{1}{2(2m+n+2)} \frac{1}{\lambda_{m+2}-\lambda_m} \int_{S^{n-1}} H^2(\omega) \,d\omega
\end{align*}
where we used \eqref{varphitilde_formula} in the last step.
\end{proof}

The following result gives an alternative characterization of half space solutions when $n=2$ and the right hand side contains a harmonic polynomial.

\begin{Lemma} \label{lemma_half_space_twodim}
Suppose that $v$ is $C^{1,1}$ and homogeneous of order $m+2$ in $\mR^2$ and solves 
\[
\Delta v = H \chi_{\{x_2 > 0\}} \text{ in $\mR^2$}, \qquad v|_{\{x_2 < 0\}} = 0
\]
where $H$ is a harmonic homogeneous polynomial of order $m$. If $m=0$ and $H \equiv a$ one has 
\[
v(x) = \frac{a}{2} (x_2)_+^2,
\]
whereas if $m \geq 1$ and $H = a r^m e^{im\theta} + b r^m e^{-im\theta}$ one has 
\begin{equation} \label{half_space_twodim}
v(x) = a(\abs{z}^2 z^m - \frac{m+1}{m+2} z^{m+2} - \frac{1}{m+2} \bar{z}^{m+2}) + b(\abs{z}^2 \bar{z}^m - \frac{1}{m+2} z^{m+2} - \frac{m+1}{m+2} \bar{z}^{m+2})
\end{equation}
where $z = x_1 + i(x_2)_+$ and $a, b$ are constants.
\end{Lemma}

\begin{Remark}
When $m=1$, \eqref{half_space_twodim} becomes 
\[
v(x) = 4 (x_2)_+^2 \left[ a (x_1 + \frac{i}{3} (x_2)_+) +  b (x_1 - \frac{i}{3} (x_2)_+) \right].
\]
When $m=2$, \eqref{half_space_twodim} becomes 
\[
v(x) = 2 (x_2)_+^2 \left[ a (3 x_1^2 + 2i x_1 (x_2)_+ - (x_2)_+^2) +  b (3 x_1^2 - 2i x_1 (x_2)_+ - (x_2)_+^2) \right].
\]
\end{Remark}

\begin{Remark}
If we consider a half space solution $v(x) = x_2^2 (a x_1 + b x_2)$ in $\{ x_2 > 0 \}$ where $a, b \in \mR$, one has $\Delta v = 2a x_1 + 6 b x_2$ in $\{ x_2 > 0 \}$. Then $H = 2 a x_1 + 6 b x_2$, and one can directly compute that 
\[
W_H(1,v) = \int_{B_1 \cap \{ x_2 > 0 \}} H v \,dx = c((2a)^2 + (6b)^2).
\]
Thus $W_H(1,v)$ only depends on the $L^2(S^1)$ norm of $H$, which is consistent with Lemma \ref{lemma_halfspace_weiss_energy}.
\end{Remark}

\begin{Remark}
In any dimension, examples of half space solutions include $(x_n)_+^3$ (for $m=1$) and $(x_n)_+^2 L(x')$ where $L$ is a harmonic homogeneous polynomial in the $x'$ variables.
\end{Remark}

\begin{proof}[Proof of Lemma \ref{lemma_half_space_twodim}]
In the case $m=0$, we may use \eqref{varphitilde_formula} (in any dimension) to have 
\[
\tilde{\varphi}(\omega) = a + A \omega \cdot \omega = (a \mathrm{Id} + A) \omega \cdot \omega, \qquad \omega \in S^{n-1},
\]
where $a$ is a constant and $A$ is a symmetric trace-free matrix. The conditions $\tilde{\varphi}|_{\p S^{n-1}_+} = \p_{\nu} \tilde{\varphi}|_{\p S^{n-1}_+} = 0$ mean that $a \mathrm{Id} + A|_{\{ e_n^{\perp} \}} = 0$ and $(a \mathrm{Id} + A) e_n = \lambda e_n$ for some constant $\lambda$. Thus $\tilde{\varphi}(\omega) = \lambda \omega_n^2$, which yields the half space solution $v = \lambda (x_n)_+^2$.

In the case $m \geq 1$ the half space solutions correspond to functions $\varphi(\omega)$ in $S^{n-1}_+$ satisfying 
\[
\tilde{\varphi}(\omega) = \frac{1}{\lambda_{m+2}-\lambda_m} H(\omega) + Q(\omega), \qquad \tilde{\varphi}|_{\p S^{n-1}_+} = \p_{\nu} \tilde{\varphi}|_{\p S^{n-1}_+} = 0,
\]
where $Q$ is a spherical harmonic of degree $m+2$. Let us next determine such functions when $n=2$ and $m \geq 1$. Then, writing $\theta$ for the angle corresponding to $\omega \in S^1$, 
\[
\varphi(\theta) = a e^{im\theta} + b e^{-im\theta} + c e^{i(m+2)\theta} + d e^{-i(m+2)\theta}.
\]
The conditions $\varphi(0) = \varphi(\pi) = 0$ give the equation 
\[
a+b+c+d=0,
\]
whereas the conditions $\varphi'(0) = \varphi'(\pi) = 0$ give the equation 
\[
m(a-b) + (m+2)(c-d) = 0.
\]
We obtain $d = c + \frac{m}{m+2}(a-b)$ and $a+b+2c+\frac{m}{m+2}(a-b)=0$. This leads to 
\[
c = -\frac{m+1}{m+2} a - \frac{1}{m+2} b, \qquad d = - \frac{1}{m+2} a - \frac{m+1}{m+2} b.
\]
Thus 
\begin{multline*}
\varphi(\theta) = a \left[ e^{im\theta} -\frac{m+1}{m+2} e^{i(m+2)\theta} - \frac{1}{m+2} e^{-i(m+2)\theta} \right] \\
 + b \left[ e^{-im\theta} - \frac{1}{m+2} e^{i(m+2)\theta} - \frac{m+1}{m+2} e^{-i(m+2)\theta} \right].
\end{multline*}
This gives the required formula for $v(r\omega) = r^{m+2} \varphi(\omega)$ for $\omega \in S^{n-1}_+$.
\end{proof}


\subsection{Classification of blowups in two dimensions}

The following result, which states that there are no blowup solutions supported in sectors of angle $\neq \pi$, is crucial in characterizing blowup solutions in two dimensions.

\begin{Lemma} \label{lemma_twodim_sector}
Let $\theta_0 \in (0,2\pi) \setminus \{ \pi \}$, let $C = \{ r e^{i\theta} \,:\, r > 0, \ \theta \in (0,\theta_0) \}$, and suppose that $u$ solves 
\[
\Delta u = H \text{ in $C$}, \qquad u|_{\theta=0} = \p_{\nu} u|_{\theta=0} = 0, \qquad u|_{\theta=\theta_0} = \p_{\nu} u|_{\theta=\theta_0} = 0
\]
where $H$ is a harmonic homogeneous polynomial of degree $m$ in $\mR^2$. Then $u \equiv 0$ and $H \equiv 0$.
\end{Lemma}
\begin{proof}
Let $u_0$ be the solution of $\Delta u_0 = H$ in $\mR^2_+$ with $u$ and $\p_n u$ vanishing on $\{ x_2 = 0 \}$ given in Lemma \ref{lemma_half_space_twodim}. As in the proof of Lemma \ref{lemma_half_space_twodim} we replace $u_0$ by its antipodally even/odd extension if $m$ is even/odd. Then $\Delta u_0 = H$ in $\mR^2$. If $m=0$ then $u_0 = \frac{a}{2} x_2^2$, whereas if $m \geq 1$ the function $u_0$ has the form 
\begin{multline*}
u_0(r e^{i\theta}) = a r^{m+2} \left[ e^{im\theta} -\frac{m+1}{m+2} e^{i(m+2)\theta} - \frac{1}{m+2} e^{-i(m+2)\theta} \right] \\
 + b r^{m+2 }\left[ e^{-im\theta} - \frac{1}{m+2} e^{i(m+2)\theta} - \frac{m+1}{m+2} e^{-i(m+2)\theta} \right].
\end{multline*}

Since $\Delta (u-u_0) = 0$ in $C$ and $u-u_0$ and $\p_n(u-u_0)$ vanish on $\{ \theta=0 \}$, the unique continuation principle implies that $u=u_0$ in $C$. It follows that $u_0$ and $\p_{\nu} u_0$ must vanish when $\theta=\theta_0 \neq \pi$. If $m = 0$, so that $u_0 = \frac{a}{2} x_2^2$, this is only possible if $a = 0$. Thus $u \equiv 0$ and $H \equiv 0$.

Let us now assume that $m \geq 1$. Writing $z = e^{i\theta_0}$, the condition that $u_0$ and $\p_{\nu} u_0$ vanish when $\theta=\theta_0$ means that 
\begin{gather*}
a \left[ z^m -\frac{m+1}{m+2} z^{m+2} - \frac{1}{m+2} z^{-m-2} \right]  + b \left[ z^{-m} - \frac{1}{m+2} z^{m+2} - \frac{m+1}{m+2} z^{-m-2} \right] = 0, \\
a \left[ m z^m - (m+1) z^{m+2} + z^{-m-2} \right]  + b \left[ -m z^{-m} -  z^{m+2} + (m+1) z^{-m-2} \right] = 0.
\end{gather*}
This pair of equations can be written in matrix form as $A(z) \left( \begin{array}{c} a \\ b \end{array} \right) = 0$, and there is a nontrivial solution iff $\det(A(z)) = 0$. Multiplying the first row of $A(z)$ by $m+2$ and adding the second row to the first row, we have that $\det(A(z))$ is a multiple of 
\[
\det \left( \begin{array}{cc} (2m+2) (z^m - z^{m+2})  & 2 (z^{-m} - z^{m+2})  \\
 m z^m - (m+1) z^{m+2} + z^{-m-2} & -m z^{-m} -  z^{m+2} + (m+1) z^{-m-2} \end{array} \right).
\]
After multiplying the second row by $2$ and subtracting the first row from the second, the determinant is a multiple of 
\[
\det \left( \begin{array}{cc} (2m+2) (z^m - z^{m+2})  & 2 (z^{-m} - z^{m+2})  \\
 -2 (z^m - z^{-m-2}) & -(2m+2) (z^{-m} - z^{-m-2}) \end{array} \right).
\]
The last determinant is equal to 
\begin{align*}
 &-4(m+1)^2(1 - z^{-2} - z^{2} + 1) + 4(1 - z^{2m+2} - z^{-2m-2} + 1) \\
 &= -4( (z^{m+1} - z^{-m-1})^2 - (m+1)^2 (z - z^{-1})^2 ).
\end{align*}
Recalling that $z = e^{i \theta_0}$, the last expression vanishes iff 
\[
\sin \theta_0 = \pm \frac{1}{m+1} \sin((m+1)\theta_0).
\]
This equation was already analyzed in \cite[Section 5]{PaivarintaSaloVesalainen} where it was proved that there are no solutions $\theta_0 \in (0,\pi)$. Since $\sin(\beta+\pi) = -\sin(\beta)$, it follows that there are no solutions $\theta_0 \in (\pi, 2\pi)$ either. This shows that when $\theta_0 \in (0,2\pi) \setminus \{\pi\}$ the determinant of $A(z)$ cannot vanish, and consequently one must have $u \equiv 0$ and $H \equiv 0$.
\end{proof}

\begin{Remark}
The conclusion of Lemma \ref{lemma_twodim_sector} fails if one drops the assumption that $H$ is harmonic. For instance, if $H = 2 \abs{x}^2$ then one has a solution $u = x_1^2 x_2^2$ in the first quadrant $C = \{ x_1, x_2 > 0 \}$ that vanishes to second order on $\p C$.
\end{Remark}

\begin{Remark}
If $0 < \theta_0 < \pi$, one can give an alternative proof of Lemma \ref{lemma_twodim_sector} as in \cite{PaivarintaSaloVesalainen} by integrating the equation $\Delta u = H \chi_C$ against harmonic functions that decay exponentially in $C$ and by using the Laplace transform.
\end{Remark}

The next result characterizes all blowup limits in two dimensions. Heuristically, one can think of the classification as follows:
\begin{itemize}
\item 
solutions with $\supp(v) = \emptyset$ correspond to boundary points where $D$ is thin (e.g.\ outward cusps);
\item 
solutions with $\supp(v)$ a half space correspond to points where $\p D$ is smooth;
\item 
solutions with $\supp(v) = \mR^2$ correspond to boundary points where $D$ has thin complement (e.g.\ inward cusps).
\end{itemize}
If $\p D$ is assumed to be Lipschitz then cusps are automatically excluded. The classification rules out the possibility that $\supp(v)$ is a sector of angle $\neq \pi$, which would correspond to points where $\p D$ would be Lipschitz but not $C^1$.

\begin{Theorem} \label{thm_blowup_classification}
Suppose that $v \in C^{1,1}_{\mathrm{loc}}(\mR^2)$ is homogeneous of degree $m+2$ and solves the equation 
\[
\Delta v = H \chi_{\{ v \neq 0 \}} \text{ in $\mR^2$}
\]
where $H$ is a harmonic homogeneous polynomial of degree $m$. Then one of the following holds:
\begin{enumerate}
\item[(a)] 
$\supp(v) = \emptyset$ and $v \equiv 0$ (so also $H \equiv 0$).
\item[(b)] 
$\supp(v)$ is a half space and $v$ is a polynomial of order $m+2$ having, after a rotation, the form given in Lemma \ref{lemma_half_space} or Lemma \ref{lemma_half_space_twodim}; or 
\item[(c)]
$\supp(v) = \mR^2$ and $v = \frac{1}{4m+4} \abs{x}^2 H + w$ where $w$ is a harmonic homogeneous polynomial of degree $m+2$.
\end{enumerate}
\end{Theorem}
\begin{proof}
Since $v$ is homogeneous, $\supp(v)$ is a closed conic set in $\mR^2$. If $\supp(v) = \emptyset$ then $v \equiv 0$ and $H \equiv 0$. If $\supp(v) = \mR^2$, then we have $\Delta v = H$ in $\mR^2$. The function $\frac{1}{4m+4} \abs{x}^2 H$ solves this equation, and hence $v = \frac{1}{4m+4} \abs{x}^2 H + w$ where $w$ is a harmonic homogeneous polynomial of degree $m+2$.

Let us now assume that $\supp(v) \neq \emptyset$ and $\supp(v) \neq \mR^2$. Then in particular $H \not\equiv 0$. Choose some $\omega_0 \in \supp(v) \cap S^1$ and let $I$ be the maximal closed connected arc on $S^1$ so that $\omega_0 \in I$ and $I \subset \supp(v)$. Since $v$ is $C^{1,1}$, it follows that $v$ and $\p_{\theta} v$ must vanish at the endpoints of $I$. Letting $C = \{ r \omega \,:\, r > 0, \ \omega \in I \}$ we see that $v$ satisfies 
\[
\Delta v = H \text{ in $C$}, \qquad v|_{\p C} = \p_{\nu} v|_{\p C} = 0.
\]
By Lemma \ref{lemma_twodim_sector} we know that $C$ must be a half space and thus $\supp(v)$ contains a closed half space. If there were some $\omega_1$ in $\supp(v)$ which is not in this half space, one could repeat the argument above and obtain a solution in a sector of angle $\theta_0$ with $0 < \theta_0 < \pi$. Another application of Lemma \ref{lemma_twodim_sector} would lead to a contradiction. This shows that $\supp(v)$ is indeed a half space.
\end{proof}

\section{Nondegeneracy and weak flatness} \label{sec_nondegeneracy}

Throughout this section, we will make the following standing assumptions. Let $q \in C^{\alpha}_{\mathrm{loc}}(B_2)$ and let $u \in C^{1,1}_{\mathrm{loc}}(B_2)$ solve 
\[
(\Delta + q) u = f \chi_D \text{ in $B_2$}, \qquad u|_{B_2 \setminus \ol{D}} = 0,
\]
where $D$ is a Lipschitz domain with $0 \in \p D$, and $f = P + R$ where $P \not\equiv 0$ is a homogeneous polynomial of order $m \geq 0$ and $R \in C^{\alpha}_{\mathrm{loc}}(B_2)$ satisfies $\abs{R(x)} \leq C \abs{x}^{m+\alpha}$ for some $\alpha > 0$. Note that under our assumptions, $\ol{B}_{r_0} \cap \ol{D} \subset \supp(f)$ for some $r_0 > 0$ since $P|_{S^{n-1}}$ cannot vanish in an open set.

In this setting, we say that $u_0 \in C^{1,1}(\ol{B}_1)$ is a blowup limit of $u$ at $0$ if there is a sequence $r_j \to 0$ so that the functions 
\[
u_{r_j}(x) = \frac{u(r_j x)}{r_j^{m+2}}, \qquad x \in \ol{B}_1,
\]
converge to $u_0$ in $C^1(\ol{B}_1)$.

We first prove a nondegeneracy result showing roughly that since $f$ vanishes precisely to order $m$, $u$ cannot vanish to order $> m+2$ in $\supp(u)$. 
It should be remarked that in free boundary problems nondegeneracy comes often naturally from the 
structure of the problem; e.g., in no-sign obstacle problems  a nonvanishing  right hand side is a key element in proving nondegeneracy \cite{PSU}.
In our setting, the problem does not have a variational formulation, and additionally, the r.h.s. may have varying vanishing order, depending on the free boundary point. Help comes from Lipschitz assumption of the free boundary, that does not allow solutions to decay faster than expected. These ideas are encoded in the proof given below.

\begin{Lemma} \label{lemma_nondegeneracy}
For any $\eps$ with $0 < \eps < 1$ there is $c_{\eps} > 0$ so that 
\[
\sup_{B(x, \eps \abs{x})} \abs{u} \geq c_{\eps} \abs{x}^{m+2} \text{ whenever $x \in \ol{D} \cap \ol{B}_{r_0}$}.
\]
\end{Lemma}
\begin{proof}

The condition for $\supp(f)$ ensures that $\ol{B}_{r_0} \cap \ol{D} \subset \supp(\Delta u)$ and hence also $\ol{B}_{r_0} \cap \ol{D} \subset \supp(u)$. In particular, $\supp(u) \cap \ol{B}_{r_0} = \ol{D} \cap \ol{B}_{r_0}$. Since $D$ is a Lipschitz domain, there is $\delta > 0$ so that possibly after shrinking $r_0$ we have the following consequence of the interior cone property:
\begin{equation} \label{interior_ball}
\text{If $x \in \ol{D} \cap \ol{B}_{r_0}$, then $B(x, r) \cap \ol{D}$ contains a ball of radius $\delta r$ whenever $r \leq r_0$.}
\end{equation}

We now prove the lemma. If the claim were not true, then there would be some $\eps \in (0,1)$ so that for any $j \geq 1$ there is $x_j \in \ol{D} \cap \ol{B}_{r_0}$, $x_j \neq 0$, so that 
\begin{equation} \label{nondegeneracy_contradiction}
\sup_{B(x_j, \eps \abs{x_j})} \abs{u} < \frac{1}{j} \abs{x_j}^{m+2}.
\end{equation}
If one had $\abs{x_j} \geq c > 0$ for all $j$, then after passing to a subsequence there would be $x \in \ol{D} \cap \ol{B}_{r_0}$ with $\abs{x} \geq c$ so that $u|_{B(x, \eps\abs{x}/2)} = 0$. This would contradict the fact that $\supp(u) \cap \ol{B}_{r_0} = \ol{D} \cap \ol{B}_{r_0}$. Thus we may assume, after taking a subsequence, that $x_j \to 0$.

Let $r_j = \abs{x_j}$ and define the rescaled function 
\[
u_{r_j}(x) = \frac{u(r_j x)}{r_j^{m+2}}.
\]
Now using \eqref{interior_ball} with $x=x_j$ and $r = \eps r_j$, for any $j \geq 1$ there is $y_j$ so that 
\begin{equation} \label{nondegeneracy_interior_ball_second}
B(y_j, \delta \eps r_j) \subset B(x_j, \eps r_j) \cap \ol{D}.
\end{equation}
It also follows that $(1-\eps) r_j \leq \abs{y_j} \leq (1+\eps) r_j$. Then by \eqref{nondegeneracy_contradiction} we have 
\[
\sup_{B(y_j, \delta \eps r_j)} \abs{u} < \frac{1}{j} r_j^{m+2}.
\]
Writing $z_j = y_j/r_j$, this means that 
\[
\sup_{B(z_j, \delta \eps)} \abs{u_{r_j}} < \frac{1}{j}.
\]
By taking another subsequence we may assume that $z_j$ converges to some $z$ with $1-\eps \leq \abs{z} \leq 1+\eps$. Thus we have 
\begin{equation} \label{nondegeneracy_uj_convergence}
u_{r_j} \to 0 \text{ in $L^{\infty}(B(z, \delta \eps/2))$}.
\end{equation}
On the other hand we have 
\[
\Delta u_{r_j}(x) = r_j^{-m} \Delta u(r_j x) = -q r_j^2 u_{r_j}(x) + (P(x) + r_j^{-m} R(r_j x)) \chi_D(r_j x).
\]
Now \eqref{nondegeneracy_interior_ball_second} implies that 
\[
\Delta u_{r_j}(x) = -q r_j^2 u_{r_j}(x) + P(x) + r_j^{-m} R(r_j x) \text{ in $B(z_j, \delta \eps)$}.
\]
Since $r_j \to 0$ and $\abs{u_{r_j}} \leq C$, $\abs{R} \leq C \abs{x}^{m+\alpha}$, we have 
\[
\Delta u_{r_j} \to P \text{ in $L^{\infty}(B(z, \delta \eps/2))$}.
\]
Since $P \not\equiv 0$, this contradicts \eqref{nondegeneracy_uj_convergence} by the uniqueness of distributional limits.
\end{proof}

\begin{Remark}
        An alternative proof for the above lemma goes as follows: 
Take $B_{2r}(z) \subset D$, close to the origin, and such that $P > 2\delta > 0$ in $B_r(z)$. Hence we conclude that $f > \delta  $ in $B_r(z)$.
 This can be done due to Lipschitz regularity of $\partial D$, and that the zero set of the polynomial $P$ is a cone. By  scaling $u$ of order $(m+2)$  we have 
$\Delta u_r = f >  \delta  $, in $B_1(z/r)$. 
 L.A. Caffarelli's non-degeneracy for the obstacle type problem (see \cite[Lemma 3.1]{PSU}) applies to obtain 
 $$
 \sup_{B_{1/2}(z/r)} u_r \geq  \frac{\delta}{8n}  + u_r(z) .
 $$
 Now either  $  u_r (z) \leq  -\delta /16n $ or 
 $ u_r (z) \geq  - \delta/16n $, and in both cases 
 we have the  nondegeneracy  of $|u|$ of order $m+2$.
 \end{Remark}

Next we establish some weak flatness properties for $\p D$ at $0$ if the support of some blowup limit is a half space. 

\begin{Lemma} \label{lemma_weak_flatness}
Suppose that $u_0$ is a blowup limit at $0$ with $\supp(u_0) = \ol{\mR^n_+}$. Then for any $\delta > 0$ there is $r > 0$ so that $\p D \cap B_{r} \subset \{ \abs{x_n} \leq \delta r \}$.
\end{Lemma}
\begin{proof}
Since $u_0$ is a blowup limit, there is a sequence $(u_{r_j})$ with $r_j \to 0$ so that $u_{r_j} \to u_0$ in $C^1(\ol{B}_1)$. It is enough to show that for any $\delta > 0$ there is $j$ so that $\p D \cap B_{r_j} \subset \{ \abs{x_n} \leq \delta r_j \}$. If this is not true, then there is $\delta > 0$ so that for any $j$ there is a point $x^{(j)} \in (\p D \cap B_{r_j}) \cap \{ \abs{x_n} > \delta r_j \}$. Then at least one of the sets $J_{\pm} = \{ j \,:\, x^{(j)} \in (\p D \cap B_{r_j}) \cap \{  \pm x_n > \delta r_j \} \}$ is infinite.

Assume that $J_+$ is infinite. After passing to suitable subsequences, we have $x^{(j)} \in  (\p D \cap B_{r_j}) \cap \{  x_n > \delta r_j \}$ for all $j \geq 1$. If $m=0$, so that $u_0 = a (x_n)_+^2$ is nonvanishing in $\mR^n_+$, we can reach a contradiction simply as follows: if we take a subsequence so that $y^{(j)} = x^{(j)}/r_j$ converges to some $y \in \ol{B}_1$, then $u_{r_j}(y^{(j)}) = 0$ implies that $u_0(y) = 0$, which is impossible since $y_n \geq \delta$.

In the case $m \geq 1$ we need to be a bit more careful. We use the following consequence of the exterior cone property of Lipschitz domains: there are $r_0, \eps > 0$ and $\omega \in S^{n-1}$ so that for any $z \in \p D \cap B_{r_0}$, the ball $B(z - \rho \omega, \eps \rho)$ is outside $\ol{D}$ whenever $0 < \rho \leq \eps$. Now since $x^{(j)} \in \p D$, we may choose $\rho_j = \frac{\delta}{4} r_j$ so that for $j$ large one has 
\[
B(x^{(j)} - \rho_j \omega, \eps \rho_j) \subset \ol{D}^c.
\]
Let $y^{(j)} = \frac{x^{(j)} - \rho_j \omega}{r_j}$. Then 
\begin{equation} \label{urj_ball_vanishing}
u_{r_j}|_{B(y^{(j)}, \eps \delta/4)} = 0.
\end{equation}
Since $x^{(j)}_n > \delta r_j$, we have $y^{(j)}_n > \frac{3\delta}{4}$. After taking a subsequence we may assume that $y^{(j)} \to y$ in $\ol{B}_1$ where $y_n \geq \frac{3\delta}{4}$. From \eqref{urj_ball_vanishing} we get that 
\[
u_0|_{B(y, \eps \delta/8)} = 0.
\]
This contradicts the fact that $\supp(u_0) = \ol{\mR^n_+}$.

Now assume that $J_-$ is infinite. After passing to subsequences, we may assume that $x^{(j)} \in  (\p D \cap B_{r_j}) \cap \{  x_n < -\delta r_j \}$ for all $j \geq 1$. Write $y^{(j)} = x^{(j)}/r_j$, so that after taking a subsequence we have $y^{(j)} \to y \in \ol{B}_1$ with $y_n \leq -\delta$. This also implies that $\delta \leq \frac{\abs{x^{(j)}}}{r_j} \leq 1$. By the nondegeneracy result in Lemma \ref{lemma_nondegeneracy}, for any $\eps > 0$ there is $c_{\eps}$ so that 
\[
\sup_{B(x^{(j)}, \eps \abs{x^{(j)}})} \abs{u} \geq c_{\eps} \abs{x^{(j)}}^{m+2}.
\]
After rescaling this gives 
\[
\sup_{B(y^{(j)}, \eps)} \abs{u_{r_j}} \geq c_{\eps} \delta^{m+2}.
\]
Since $u_{r_j} \to u_0$ in $C^1(\ol{B}_1)$, we also have 
\[
\sup_{B(y, 2\eps)} \abs{u_0} \geq c_{\eps} \delta^{m+2}.
\]
But if we choose $\eps < \delta/2$ this contradicts the fact that $\supp(u_0) = \ol{\mR^n_+}$. We have thus proved the lemma.
\end{proof}

Our final lemma proves $C^1$ regularity of the free boundary of convex domains that have unique supporting planes.

\begin{Lemma} \label{lemma_convex_c1}
Let $D \subset \mR^n$ be an open bounded  set which is convex. Suppose that for any $z \in \p D \cap B_{r_0} (x^0)$, with  $x^0 \in \p D$, there is a unique supporting plane.
Then $\p D \cap B_{r_0/2} (x^0)$ is $C^1$ near $x^0$.
\end{Lemma}

\begin{proof}
It suffices to show that for each $\eps >0 $ there is $r_\eps $ such that for any $z \in \p D \cap B_{r_0/2} (x^0)$ we have a cone $C_{z,\eps }$ of opening $\pi - \eps $ and vertex $z$,   such that its   truncation
\[
C_{z,\eps }^{r_\eps } = C_{z, \eps }  \cap B_{r_\eps}(z)
\]
is inside $D$. If this fails, 
then for a given $\eps > 0$ 
we have  
$z^j \in \p D \cap B_{r_j/2} (x^0)$, and   $r_j$ such that the corresponding truncated cone is not in $D$, meaning  
\begin{equation}\label{eq:cone}
    C_{z^j,\eps }^{r_j } \setminus D \neq \emptyset. 
\end{equation}
In particular this holds for the scaled version of $D$, which we call $D_j = \frac{D - z^j}{r_j}$.
Upon letting $j \to \infty $, and hence $r_j \to 0$,  we arrive at a limit domain $D_0$. By assumptions there is also a half space representing the set $D_0$. This  does not align with the  assumption \eqref{eq:cone}, in scaled version, that $C_{z,\eps }^{1} \setminus D_j \neq \emptyset$, where now $D_j$ is an scaled version of $D$ at $x^j$. Hence we arrive at a contradiction and the lemma is proved. 
 \end{proof}

\section{Proofs of main theorems} \label{sec_fb_regularity}

We can now give the proofs of the main theorems in the introduction. We begin with the result on basic properties of blowups.

\begin{proof}[Proof of Theorem \ref{thm_fbp_steps}]
This follows by combining Lemma \ref{lemma_product_vanishing}, Theorem \ref{thm_c11_regularity}, and Lemmas \ref{lemma_weiss_blowup}, \ref{lemma_nondegeneracy} and \ref{lemma_weak_flatness}.
\end{proof}

Next we prove the classification of blowup solutions in two dimensions when the boundary is Lipschitz and one has nondegeneracy.

\begin{proof}[Proof of Theorem \ref{thm_blowup_classification_intro}]
Theorem \ref{thm_blowup_classification} ensures that the support of any blowup limit $v$ is either $\emptyset$, a half space or $\mR^2$. The case $\supp(v) = \mR^2$ is ruled out by the assumption that $D$ has Lipschitz boundary (so $v$ vanishes in some exterior cone). The case $\supp(v) = \emptyset$ is not possible  by the nondegeneracy result in Theorem \ref{thm_fbp_steps}. Thus the support of any blowup limit $v$ must be a half space.
\end{proof}

We can now prove the two-dimensional free boundary regularity result.

\begin{proof}[Proof of Theorem \ref{thm_fb_2d}]
Suppose that $\Delta u = f \chi_D$ in $B_2$ and $u=0$ outside $D$, where $f = H + R \in C^{\alpha}_{\mathrm{loc}}(B_2)$ with $H \not\equiv 0$ a harmonic homogeneous polynomial of degree $m$ and $|R(x)| \leq C |x|^{m+\alpha}$. Theorem \ref{thm_fbp_steps} gives optimal regularity, i.e.\ $u \in C^{1,1}_{\mathrm{loc}}(B_2)$ with 
\[
|u(x)| + |x| \,|\nabla u(x)| + |x|^2 \, |\nabla^2 u(x)| \leq C |x|^{m+2}.
\]
It also follows that any blowup limit $v$ of $u$ is homogeneous of degree $m+2$ and solves 
\[
\Delta v = H \chi_{ \{ v \neq 0 \} } \text{ in $\mR^n$.}
\]

Now we invoke the assumption that $H$ is harmonic and $n=2$. Theorem \ref{thm_blowup_classification_intro} ensures that the support of any blowup limit $v$ must be a half space. Suppose that $\{x \cdot e \geq 0 \}$ is the support of some blowup limit, where $e$ is a unit vector. Then the weak flatness property in Theorem \ref{thm_fbp_steps} ensures that for any $\delta > 0$ there is $r > 0$ so that $\p D \cap B_{r} \subset \{ \abs{x \cdot e} \leq \delta r \}$.

Suppose now that $\p D$ is piecewise $C^1$. Then near $0$, $\p D$ is the union of two $C^1$ arcs meeting at $0$ at some angle $\theta \in (0,2\pi)$. If $\theta \neq \pi$, one gets a contradiction with the weak flatness property above. This implies that $\theta = \pi$ and the tangent vectors of the two $C^1$ arcs are parallel at $0$. Thus $\p D$ is the graph of a $C^1$ function near $0$.

Similarly, if $D$ is convex and if one has decompositions $f = H + R$ at all points $z \in \p D$ near $0$, then one has the weak flatness property at such points $z$. Thus there is a unique supporting plane at boundary points near $0$. Lemma \ref{lemma_convex_c1} then implies that $\p D$ is $C^1$ near $0$.

\end{proof}

We also prove the result for edge points in dimension $n \geq 3$.

\begin{proof}[Proof of Theorem \ref{thm_fb_polyhedron}]
We suppose that $0 \in \p D$ is an edge point. The first two steps are the same as in two dimensions and follow from Theorem \ref{thm_fbp_steps}:
\begin{itemize}
\item 
Optimal regularity: $u \in C^{1,1}_{\mathrm{loc}}(B_2)$ with 
\[
|u(x)| + |x| \,|\nabla u(x)| + |x|^2 \, |\nabla^2 u(x)| \leq C |x|^{m+2}.
\]
\item 
Any blowup limit $v$ of $u$ is homogeneous of degree $m+2$ and solves 
\[
\Delta v = H \chi_{C} \text{ in $\mR^n$}
\]
where $C =  \supp(v)$ is a conic set.
\end{itemize}
Since $0 \in \p D$ was assumed to be an edge point, after a rotation we may arrange that the blowup of $D$ at $0$ (i.e.\ the limit of $r^{-1}(D \cap B_r)$ as $r \to 0$) is $(S \times \mR^{n-2}) \cap B_1$ where $S$ is a closed sector in $\mR^2$ with angle $\neq \pi$. Then $C \subset S \times \mR^{n-2}$, and the nondegeneracy statement in Theorem \ref{thm_fbp_steps} implies that $C = S \times \mR^{n-2}$.

Now we wish to apply the dimension reduction argument of Federer as in \cite{Weiss1999} or \cite[Lemma 10.9]{Velichkov2023}. Fix a point $e = e_n \in \p C$, so that the blowup of $C$ at $e$ is $S \times \mR^{n-2}$. Suppose that $H$ has a zero at $e$ of order $k$ (then $0 \leq k \leq m$). By using the optimal regularity result for $v$ at $e$, we have that 
\[
|v(e+z)| + |z| \,|\nabla v(e+z)| + |z|^2 \, |\nabla^2 v(e+z)| \leq C |z|^{k+2}.
\]
We take a second blowup at $e$ defined via the sequence 
\[
v_r(x) = \frac{v(e + rx)}{r^{k+2}}.
\]
By compactness there is a blowup limit $w$ of $v_r$, and by Weiss monotonicity as before $w$ is homogeneous of degree $k+2$ and satisfies 
\[
\Delta w = P \chi_{\mathrm{supp}(w)}
\]
where $H(e+z) = P(z) + O(|z|^{k+1})$. Note that $P$ is a harmonic homogeneous polynomial of order $k$ since $H$ is harmonic. Moreover, the fact that the blowup of $C$ at $e$ is $S \times \mR^{n-2}$ together with nondegeneracy imply that $\mathrm{supp}(w) = S \times \mR^{n-2}$.

The new input is to show that $w$ is independent of $x_n$. To see this, we observe that the homogeneity of $v$ gives 
\[
r^{k+1} (e+rx) \cdot \nabla v_r(x) = (e+rx) \cdot \nabla v(e+rx) = (m+2) v(e+rx) = (m+2) r^{k+2} v_r(x).
\]
Thus 
\[
e \cdot \nabla v_r(x) = (m+2) r v_r(x) - rx \cdot \nabla v_r(x).
\]
By taking the limit as $r \to 0$, one has 
\[
e \cdot \nabla w(x) = 0.
\]
Thus $w$ is indeed independent of $x_n$. Since $\chi_{\mathrm{supp}(w)}$ is also independent of $x_n$, the same is true for $P$. Thus, writing $x = (x',x_n)$, we have 
\[
\Delta w(x') = P(x') \chi_{S \times \mR^{n-3}}(x') \text{ in $\mR^{n-1}$}.
\]
Thus we have reduced matters into a problem in $\mR^{n-1}$. We can repeat this procedure until we arrive at the problem 
\[
\Delta w = P \chi_{S} \text{ in $\mR^2$}.
\]
Now the two-dimensional result, Lemma \ref{lemma_twodim_sector}, implies that $S$ must be a half-space, which contradicts our assumption that $0$ was an edge point.
\end{proof}

Finally, we prove the theorems related to inverse scattering problems.

\begin{proof}[Proof of Theorems \ref{thm_main_twodim} and \ref{thm_main_ndim}]
As discussed in the introduction (see \eqref{eq_fbp_first}), we can reduce matters to the equation 
\[
(\Delta+k^2+q)w = f \chi_D \text{ in $\Omega$}, \qquad w = 0 \text{ outside $D$},
\]
where $f := -h u_0$. The contrast $h$ is $C^{\alpha}$ near $0$ with $h(0) \neq 0$ by \eqref{nondegen}, and $u_0$ solves $(\Delta+k^2) u_0 = 0$ in $\mR^n$. Thus, at points $z \in \p D$ near $0$, one has $h(z) \neq 0$ and $u_0(z+x) = H(x) + R(x)$ where $H$ is a harmonic homogeneous polynomial of order $m$ and $|R(x)| \leq C|x|^{m+1}$ (here $H$, $R$ and $m$ depend on $z$). In fact, $H$ is the first nontrivial homogeneous polynomial in the Taylor expansion of $u_0(z+x)$, and the equation $\Delta H = 0$ is obtained from $(\Delta+k^2)u_0 = 0$ by looking at blowups of $u_0(z+x)$ of order $m+2$.

After this reduction, Theorems \ref{thm_main_twodim} and \ref{thm_main_ndim} follow from the corresponding free boundary results in Theorem \ref{thm_fb_2d} and \ref{thm_fb_polyhedron}.
\end{proof}

\begin{Remark} \label{rem_eh_assumption}
The main theorems on inverse scattering were proved under the condition \eqref{nondegen}, which is the case $l = 0$ of \cite[Assumption (a)]{ElschnerHu2}. That assumption in principle includes the case $l \geq 1$, but in that case we can show that the boundary is necessarily $C^1$ and thus the conclusion of the theorems holds automatically.

Indeed, suppose that \cite[Assumption (a)]{ElschnerHu2} holds for some $l \geq 1$ at point $O = 0$. Writing $f = q-1$, this means that for some $s, \eps > 0$ one has 
\[
f \in C^{l,s}(\ol{D \cap B_{\eps}}) \cap W^{l,\infty}(B_{\eps})
\]
and that, after renaming coordinates if necessary, $\p^{\alpha+e_n} f(0) \neq 0$ for some $\alpha$ with $|\alpha|=l-1$ (we may assume that $\p^{\beta} f(0) = 0$ for $|\beta| \leq l-1$). Since it was assumed in \cite{ElschnerHu2} that $q=1$ outside $D$, and since $f \in C^{l-1}(B_{\eps})$, it follows that $\p^{\beta} f|_{\p D \cap B_{\eps}} = 0$ for $|\beta| \leq l-1$.

Using the implicit function theorem to $\p^{\alpha} \tilde{f}$ where $\tilde{f}$ is a $C^{l}$ extension of $f$ from $\ol{D \cap B_{\eps}}$ to $B_{\eps}$, we see that there is a $C^1$ function $\eta$ near $0'$ in $\mR^{n-1}$ such that 
\[
\p^{\alpha} \tilde{f} = 0 \text{ near $0$} \quad \Longleftrightarrow \quad x_n = \eta(x').
\]
Since $\p^{\alpha} f|_{\p D \cap B_{\eps}} = 0$, it follows that $\p D$ near $0$ is contained in the graph of the $C^1$ function $\eta$. In particular, all tangent vectors of $\p D$ at $0$ must be contained in a fixed hyperplane, so $\p D$ cannot have a corner point at $0$ if $n=2$ or an edge point at $0$ if $n=3$. This shows that if \cite[Assumption (a)]{ElschnerHu2} holds for some $l \geq 1$, then the boundary has to be $C^1$.
\end{Remark}

\section*{Acknowledgments}

\noindent  
Salo was partly supported by the Research Council of Finland (Centre of Excellence in Inverse Modelling and Imaging and FAME Flagship, grants 353091 and 359208). Shahgholian was supported by Swedish Research Council (grant no. 2021-03700).

\section*{Declarations}

\noindent {\bf  Data availability statement:} All data needed are contained in the manuscript.

\medskip
\noindent {\bf  Funding and/or Conflicts of interests/Competing interests:} The authors declare that there are no financial, competing or conflict of interests.

\end{document}